\documentclass[11pt]{amsart}
\usepackage{geometry}                
\usepackage{graphicx}
\usepackage{amssymb}
\usepackage{epstopdf}
\title{4-dimensional riemannian manifolds with a harmonic 2-form of constant length}

\begin{document}
\author{Inyoung Kim}
\maketitle
  \begin{abstract}
  It was shown by Seaman that if a compact, oriented 4-dimensional riemannian manifold $(M, g)$ of positive
 sectional curvature admits a harmonic 2-form of constant length, its intersection form is definite and such a harmonic form is unique up to constant multiples. 
 In this paper, we show that such a manifold is diffeomorphic to 
 $\mathbb{CP}_{2}$ with a slightly weaker curvature hypothesis and there is an infinite dimensional moduli space of such metrics near 
 the Fubini-Study metric on  $\mathbb{CP}_{2}$.
 We discuss some of conditions  which can be added in order to get the Fubini-Study metric
 up to diffeomorphisms and rescaling.

\end{abstract}
\maketitle
\section{\large\textbf{Introduction}}\label{S:Intro}
Seaman showed that if a compact, oriented 4-dimensional manifold of positive sectional curvature
admits a harmonic 2-form of constant length, then this harmonic 2-form is either self-dual or anti-self-dual [17]. 
This result was improved by Noronha [15] and Costa and Ribeiro [4] by using weaker curvature hypothesis. 
To describe these results, let $P\subset T_{p}M$ be a plane at a tangent space of $p\in M$.
Then we define biorthogonal sectional curvature by 
\[K^{\perp}(P)=\frac{K(P)+K(P^{\perp})}{2}.\]
Seaman's theorem was improved in [15] by assuming $K^{\perp}>0$ rather than positive sectional curvature
and in [4] by assuming $K^{\perp}<\frac{s}{4}$ and positive scalar curvature condition.

In this paper, we relate these results with an almost-K\"ahler structure. 
For this, we explain an oriented, compact 4-dimensional riemannian manifold $(M, g)$ with a self-dual harmonic 2-form 
with length $\sqrt{2}$ defines an almost-K\"ahler structure. 
Then by using Liu's theorem [13], we show that if a compact, 4-dimensional manifold admits a harmonic 2-form of constant length
and has $K^{\perp}>0$, then $M$ is diffeomorphic to $\mathbb{CP}_{2}$.
This kind of argument was already used in B\"ar [2] in somewhat different case. 
Since we have not found a paper which explicitly mention theorems in the form of this paper, 
we provide some detail and discuss the existence of metrics which has a harmonic 2-form of constant length with positive sectional curvature on  $\mathbb{CP}_{2}$. 

\vspace{20pt}
$\mathbf{Acknowledgement}$: The author is very thankful to Prof. Claude LeBrun for suggesting the problem and helpful discussions. 
\vspace{ 50pt} 

\section{\large\textbf{Seaman's result and its improvements}}\label{S:Intro}
\newtheorem{Theorem}{Theorem}
First, we state results in [17], [15], [4]. In this paper, we only consider smooth, connected 4-dimensional manifolds. 
\begin{Theorem}
(Seaman [17]) Let $(M^{4}, g)$ be a compact, oriented riemannian manifold of positive sectional curvature. 
Then up to constant multiples, $M$ has at most one harmonic 2-form of constant length. 
If $M$ admits a harmonic 2-form of constant length, then the intersection form of $M$ is definite. 
\end{Theorem}

\begin{Theorem}
(Noronha [15])
Let $(M^{4}, g)$ be a compact, oriented riemannian manifold and suppose $K^{\perp}>0$. 
If $M$ admits a harmonic 2-form of constant length, then the intersection form of $M$ is definite. 
\end{Theorem}

\begin{Theorem}
(Costa - Ribeiro [4])
Let $(M^{4}, g)$ be a compact, oriented riemannian manifold.
Suppose $K^{\perp}<\frac{s}{4}$, where $s$ is the scalar curvature. 
If $M$ admits a harmonic 2-form of constant length, then the intersection form of $M$ is definite. 
\end{Theorem}

\vspace{20pt}
We prove this result in a slightly different way and improve the conclusion. 
We mention our main theorems.

\begin{Theorem}
Let $(M^{4}, g)$ be a compact, riemannian manifold. Suppose $K^{\perp}>0$ or $K^{\perp}<\frac{s}{4}$ on $M$. 
If $(M, g)$ admits a harmonic 2-form of constant length, then $M$ is orientable and 
there is a choice of an orientation such that $(M, g)$ defines an almost-K\"ahler structure.
Moreover, $M$ is diffeomorphic to $\mathbb{CP}_{2}$.

\end{Theorem}

\begin{Theorem}
There is an infinite dimensional moduli space of metrics on $\mathbb{CP}_{2}$
which have a harmonic 2-form of constant length and have $K^{\perp}>0$. 
\end{Theorem}
\vspace{20pt}

Let $(M^{4}, g)$ be an oriented 4-dimensional riemannian manifold.
We identify vectors and forms canonically using a metric $g$. 
Let us fix one point $p$ on $M$. 
Note that $W_{\pm}=R-\frac{s}{12}I_{\pm} : \Lambda^{\pm}\to\Lambda^{\pm}$ is symmetric since $R$ is symmetric. 
Thus there is an orthonormal basis of eigenvectors. Let $\alpha_{i}^{+}$ for $i=1, 2, 3$ be such eigenvectors of $W_{+}$ at $p$ 
with eigenvalues $\lambda_{i}^{+}$ and 
$\alpha_{i}^{-}$ for $i=1, 2, 3$ be eigenvectors of $W_{-}$ with eigenvalues $\lambda_{i}^{-}$ at $p$. 
We assume $\lambda_{1}^{+}\leq \lambda_{2}^{+}\leq \lambda_{3}^{+}$. 
Similarly, for $W_{-}$, we assume $\lambda_{1}^{-}\leq \lambda_{2}^{-}\leq \lambda_{3}^{-}$.
From the First Bianchi identity, we get $W_{\pm}$ are trace-free, and therefore we have $\lambda_{1}^{+}+ \lambda_{2}^{+}+\lambda_{3}^{+}=0$
and $\lambda_{1}^{-}+ \lambda_{2}^{-}+\lambda_{3}^{-}=0$.
Thus, we have 
\[\lambda_{1}^{\pm}\leq 0, \lambda_{3}^{\pm}\geq 0\]
and
\[-\frac{1}{2}\lambda_{1}^{\pm}\leq \lambda_{3}^{\pm}\leq -2\lambda_{1}^{\pm}.\]

Let us consider the following expressions:
\[K^{\perp}_{1}=\min\limits_{P\subset T_{p}M}\{\frac{K(P)+K(P^{\perp})}{2}\},\]
\[K^{\perp}_{3}=\max\limits_{P\subset T_{p}M}\{\frac{K(P)+K(P^{\perp})}{2}\}.\]
We provide a proof of the following Lemma in detail, which is stated in [4]. 

\newtheorem{Lemma}{Lemma}
\begin{Lemma}
\[2K^{\perp}_{1}=\frac{s}{6}+\lambda_{1}^{+}+\lambda_{1}^{-}\]
\[2K^{\perp}_{3}=\frac{s}{6}+\lambda_{3}^{+}+\lambda_{3}^{-}\]

\end{Lemma}

\begin{proof}
Any plane $P, P^{\perp}\subset T_{p}M$ can be represented by an orthonormal frame $\{e_{1}, e_{2}, e_{3}, e_{4}\}$
such that $P=span\{e_{1}, e_{2}\}$ and $P^{\perp}=span\{e_{3}, e_{4}\}$. 
Since we identify vectors and forms using $g$, we get the corresponding orthonormal coframe, which we still denote by $\{e_{1}, e_{2}, e_{3}, e_{4}\}$.
We suppose $\{e_{1}, e_{2}, e_{3}, e_{4}\}$ is positively oriented. Then we have 
\[K(P)+K(P^{\perp})\]
\[=\frac{s}{6}+W_{+}(\frac{e_{1}\wedge e_{2}+e_{3}\wedge e_{4}}{\sqrt{2}}, \frac{e_{1}\wedge e_{2}+e_{3}\wedge e_{4}}{\sqrt{2}})
+W_{-}(\frac{e_{1}\wedge e_{2}-e_{3}\wedge e_{4}}{\sqrt{2}}, \frac{e_{1}\wedge e_{2}-e_{3}\wedge e_{4}}{\sqrt{2}}).\]
Since $\frac{e_{1}\wedge e_{2}+e_{3}\wedge e_{4}}{\sqrt{2}}=a_{1}\alpha_{1}^{+}+a_{2}\alpha_{2}^{+}+a_{3}\alpha_{3}^{+}$  for some $a_{i}\in \mathbb{R}$ 
such that  $a_{1}^{2}+a_{2}^{2}+a_{3}^{2}=1$, 
 we have 
\[W_{+}(\frac{e_{1}\wedge e_{2}+e_{3}\wedge e_{4}}{\sqrt{2}}, \frac{e_{1}\wedge e_{2}+e_{3}\wedge e_{4}}{\sqrt{2}})
=\lambda_{1}^{+}a_{1}^{2}+\lambda_{2}^{+}a_{2}^{2}+\lambda_{3}^{+}a_{3}^{2}\]
\[\geq \lambda_{1}^{+}(a_{1}^{2}+a_{2}^{2}+a_{3}^{2})=\lambda_{1}^{+}.\]
Thus, we get 
\[K_{1}^{\perp}\geq \frac{s}{6}+\lambda_{1}^{+}+\lambda_{1}^{-}.\]

Note that $\alpha_{1}^{+}$ is self-dual and $\alpha_{1}^{-}$ is anti-self-dual with $|\alpha_{1}^{+}|=|\alpha_{1}^{-}|=1$. 
Then we have 
\[(\alpha_{1}^{+}+\alpha_{1}^{-})\wedge (\alpha_{1}^{+}+\alpha_{1}^{-})=|\alpha_{1}^{+}|^{2}-|\alpha_{1}^{-}|^{2}=0.\]
On 2-forms on a 4-manifold, this implies $\alpha_{1}^{+}+\alpha_{1}^{-}$ is decomposable.
Then, there exist orthonormal vectors $\{f_{1}, f_{2}\}$ such that 
\[\alpha_{1}^{+}+\alpha_{1}^{-}=\sqrt{2}f_{1}\wedge f_{2}.\]
Thus, there exists a positively oriented orthonormal coframe $\{f_{1}, f_{2}, f_{3}, f_{4}\}$ at $p$ such that 
\[\alpha_{1}^{+}=\frac{1}{\sqrt{2}}(f_{1}\wedge f_{2}+f_{3}\wedge f_{4})\]
and 
\[\alpha_{1}^{-}=\frac{1}{\sqrt{2}}(f_{1}\wedge f_{2}-f_{3}\wedge f_{4}).\]
By considering $P=span\{f_{1}, f_{2}\}$, we get $K_{1}^{\perp}\leq \frac{s}{6}+\lambda_{1}^{+}+\lambda_{1}^{-}.$
Thus, we get $2K^{\perp}_{1}=\frac{s}{6}+\lambda_{1}^{+}+\lambda_{1}^{-}$. Then the claim follows. 

In a similar way, we can show that $2K^{\perp}_{3}=\frac{s}{6}+\lambda_{3}^{+}+\lambda_{3}^{-}$.

\end{proof}

\newtheorem{Corollary}{Corollary}

\begin{Corollary}
Let $(M^{4}, g)$ be an oriented  riemannian manifold. Suppose $K_{1}^{\perp}>0$. 
Then  $K_{3}^{\perp}<\frac{s}{4}$. 
\end{Corollary}
\begin{proof}
Note that  $K_{1}^{\perp}>0$ implies $\frac{s}{6}+\lambda_{1}^{+}+\lambda_{1}^{-}>0$. 
Since we have
\[2\lambda_{1}^{\pm}\leq -\lambda_{3}^{\pm}, \]
we get 
\[\frac{s}{12}+\lambda_{1}^{\pm}\leq \frac{s}{12}-\frac{\lambda_{3}^{\pm}}{2}.\]
Thus, we have 
\[0<\frac{s}{6}+\lambda_{1}^{+}+\lambda_{1}^{-}<\frac{s}{6}-\frac{\lambda_{3}^{+}}{2}-\frac{\lambda_{3}^{-}}{2}.\]
\end{proof}

\begin{Corollary}
Let $(M^{4}, g)$ be an oriented riemannian manifold. If $K^{\perp}>0$ or $K^{\perp}<\frac{s}{4}$, 
then the scalar curvature $s$ is positive. 
\end{Corollary}
\begin{proof}
Using an orthonormal basis, we can check directly, that $K^{\perp}>0$ implies $s>0$. 
Note that $K^{\perp}<\frac{s}{4}$ is equivalent to $\frac{2s}{3}-2\lambda_{3}^{+}-2\lambda_{3}^{-}>0$. 
Since $\lambda_{3}^{\pm}\geq 0$, we get if $K^{\perp}<\frac{s}{4}$, then $s>0$. 
\end{proof}

Let us write $r_{i}^{\pm}=\frac{s(p)}{3}-2\lambda_{i}^{\pm}$. 
\begin{Lemma}
Let $(M^{4}, g)$ be an oriented riemannian manifold. Suppose $K_{3}^{\perp}<\frac{s}{4}$. 
Then at a point $p$, we have $0<r_{i}^{+}+r_{j}^{-}$ for $i, j=1, 2, 3$. 

\end{Lemma}
\begin{proof}
Note that the condition $K_{3}^{\perp}<\frac{s}{4}$ is equivalent to the following inequality,
\[0<\frac{2s}{3}-2\lambda_{3}^{+}-2\lambda_{3}^{-}.\]
Thus, at a point $p$, we have either $\frac{s}{3}-2\lambda_{3}^{+}>0$ or $\frac{s}{3}-2\lambda_{3}^{-}>0$. 
Also we have  $0<r_{i}^{+}+r_{j}^{-}$ for $i, j=1, 2, 3$.

\end{proof}

\begin{Lemma}
Let $(M^{4}, g)$ be a compact, oriented riemannian manifold such that $K_{3}^{\perp}<\frac{s}{4}$. 
If $(M, g)$ admits a harmonic 2-form of constant length, then either $\frac{s}{3}-2W_{+}$ or $ \frac{s}{3}-2W_{-}$ is positive on $M$. 
\end{Lemma}
\begin{proof}
We fix a point $p$. 
Let $r_{i}^{+}(p)$ be eigenvalues of $\frac{s}{3}-2W_{+}|_{p}$ and $r_{i}^{-}(p)$ be eigenvalues of $\frac{s}{3}-2W_{-}|_{p}$.
We assume $r_{3}^{+}(p)\leq r_{2}^{+}(p)\leq r_{1}^{+}(p)$ and  $r_{3}^{-}(p)\leq r_{2}^{-}(p)\leq r_{1}^{-}(p)$.
We regard $r_{i}^{\pm}$ as continuous functions on $M$.
Then Lemma 2 tells us that $0<r_{i}^{+}+r_{j}^{-}$ at any point on $M$. 
We claim at least one of $r_{i}^{\pm}$ be nonpositive at any point on $M$.  If all $r_{i}^{\pm}$ are positive at $q$, then $\frac{s}{3}-2W$ is a positive operator at $q$.
Let $\omega$ be a harmonic 2-form of constant length.
Then at a point $q$, we have 
\[0=<\Delta\omega, \omega>=\frac{1}{2}\Delta|\omega|^{2}+|\nabla\omega|^{2}-2W(\omega, \omega)+\frac{s}{3}|\omega|^{2}.\]
Since $\frac{s}{3}-2W$ is positive at $p$ and $|\omega|^{2}$ is constant, we get $\omega=0$ at $p$, which is a contradiction.

Thus, one of $r_{3}^{+}(p)$ or $r_{3}^{-}(p)$ is nonpositive. Let us assume $r_{3}^{+}(p)\leq 0$ and then we have $r_{i}^{-}(p)>0$ for $i=1, 2, 3$. 
We claim $\frac{s}{3}-2W_{-}$ is positive on $M$. 
We consider the following closed subset of $M$. 
\[A_{i}=\{x\in M|r^{-}_{i}(x)=0\}.\]
On a riemannian manifold $(M, g)$ , distance $d_{g}$ between two points is defined and using this, metric topology is given and this topology is the same with 
the  topology on $M$ given by the differentiable structure. 
Since $A_{i}$ is compact and $d_{g}(p, \cdot)$ is a continuous function on $A_{i}$, there exists $c_{i}\in A_{i}$ such that $d_{g}(p, c_{i})$ takes the minimum $a_{i}$. 

If $\frac{s}{3}-2W_{-}$ is not positive, at least one of $A_{i}$ is not empty. 
We denote $a_{m}$  be the $min\{a_{i}\}$, where $A_{i}$ is non-empty. 
Note that at $c_{m}$, we have $r_{m}^{-}(c_{m})=0$ Thus, we have $r_{i}^{+}(c_{m})>0$. 
Since $r_{i}^{+}$ are continuous, there exists a neighborhood $U$ of $c_{m}$ such that $r_{i}^{+}(U)>0$. 
Note that by shrinking $U$ if necessary, we can assume $p\in U^{c}$.
Then we take a ball $B(c_{m}, \epsilon)=\{x|d_{g}(x, c_{m})<\epsilon\}$ such that $\overline{B(c_{m}, \epsilon)}\subset U$. 
Since $M$ is compact, $(M, d_{g})$ is complete. 
Thus, there exists a curve $\gamma$ joining $c_{m}$ and $p$ such that the length of $\gamma$ achieves $d_{g}(p, c_{m})$. 
Note that $\gamma$ intersects with the boundary of $\overline{B(c_{m}, \epsilon)}$. 
Let $d_{m}$ be the intersection point. 
Since $d_{g}(p, d_{m})<a_{m}$, we have $r_{i}^{-}(d_{m})>0$ for $i=1, 2, 3$. For otherwise, that is, if $r_{i}^{-}(d_{m})\leq 0$, then by intermediate value theorem, 
there is a point $q\in \gamma |_{[p, d_{m}]}$ such that $r_{i}(q)=0$, which contradicts the definition of $c_{m}$. 
Thus, at $d_{m}$, we have $r_{i}^{\pm}(d_{m})>0$, which gives a contradiction under the existence of a harmonic 2-form of constant length. 
Thus, $\frac{s}{3}-2W_{-}$ is positive on $M$. 
If we assume $r^{-}_{3}(p)\leq 0$ instead, then we get $\frac{s}{3}-2W_{+}$ is positive on $M$. 
\end{proof}

\newtheorem{Proposition}{Proposition}
\begin{Proposition}
Let $(M^{4}, g)$ be a compact, oriented riemannian manifold with $K_{3}^{\perp}<\frac{s}{4}$. 
If $(M, g)$ admits a harmonic 2-form $\omega$ of constant length, then this form is either self-dual or anti-self-dual.
Moreover, the intersection form of $M$ is definite. 
\end{Proposition}
\begin{proof}
For a 2-form $\alpha$ on a 4-manifold, its self-dual part is given by 
$\alpha^{+}=\frac{\alpha+*\alpha}{2}$ and its anti-self-dual part is given by $\alpha^{-}=\frac{\alpha-*\alpha}{2}$. 
Since the Hodge star operator sends harmonic forms to harmonic forms, we get $\alpha^{\pm}$ are harmonic if $\alpha$ is a harmonic 2-form. 

By Lemma 4, either $\frac{s}{3}-2W_{+}$ or $ \frac{s}{3}-2W_{-}$ is positive.
Suppose $ \frac{s}{3}-2W_{-}$ is positive.
From the Weitzenb\"ock formula about an anti-self-dual 2-form $\omega_{-}$, which is harmonic, we have 
\[0=\Delta\omega_{-}=\nabla^{*}\nabla\omega_{-}-2W_{-}(\omega_{-}, \cdot)+\frac{s}{3}\omega_{-}.\]
If we take an inner product with $\omega_{-}$, we get
\[0=\frac{1}{2}\Delta|\omega_{-}|^{2}+|\nabla\omega_{-}|^{2}-2W(\omega_{-}, \omega_{-}) +\frac{s}{3}|\omega_{-}|^{2}.\]
Since $\frac{s}{3}-2W_{-}$ is a positive operator, by taking an integration 
we get
\[\int_{M} \frac{1}{2}\Delta|\omega_{-}|^{2}+|\nabla\omega_{-}|^{2}=\int |\nabla\omega_{-}|^{2}<0,\]
which is a contradiction. Thus, we get $\omega_{-}=0$. 
The same argument shows that there are no anti-self-dual harmonic 2-forms on $M$. 
\end{proof}

\newtheorem{Remark}{Remark}
\begin{Remark}
If $(M^{4}, g)$ is a compact, oriented manifold of positive sectional curvature, 
then by Synge's theorem, $M$ is simply connected. 
Then by results of Donaldson and Freedman [5], [6], 
Seaman's result implies that
if such $M$ admits a harmonic 2-form of constant length, then it is homeomorphic to 
\[\mathbb{CP}_{2}\#\cdot\cdot\cdot \#\mathbb{CP}_{2}.\]
We show that $M$ is diffeomorphic to $\mathbb{CP}_{2}$ in the next section. 
\end{Remark}

\begin{Remark}
Note that in [17], not only $r_{i}^{+}+r_{j}^{-}>0$ but also $r_{i}^{+}+r_{j}^{+}>0$ for $i\neq j$ and $r_{i}^{-}+r_{j}^{-}>0$ for $i\neq j$ were used. 
Note that for Lemma 3, we only need  $r_{i}^{+}+r_{j}^{-}>0$. 
This is the main difference between Seaman's and ours. 
Though positive sectional curvature is assumed in [17],  we note that the method of proof in [17] still works under the hypothesis
of $K^{\perp}>0$. Namely, $r$ in (17), [16] is positive if $K^{\perp}>0$.
 \end{Remark}

\vspace{50pt}

\section{\large\textbf{Main results}}\label{S:Intro}
Let $(M^{4}, g)$ be a compact, oriented riemannian manifold and 
suppose a self-dual harmonic 2-form $\omega$ with constant length $\sqrt{2}$ is given. 
Then we have $\omega=*\omega$ and $\Delta\omega=0$.
Since $\Delta=dd^{*}+d^{*}d$ and $d^{*}$ is the adjoint of $d$, we have 
\[\int_{M}<\Delta\omega, \omega>d\mu=\int_{M}<(dd^{*}+d^{*}d)\omega, \omega>d\mu=\int_{M}<d^{*}\omega, d^{*}\omega>+<d\omega, d\omega>d\mu.\]
Thus, if $\Delta\omega=0$, we get $d^{*}\omega=d\omega=0$. 
Also note that $\omega\wedge\omega=|\omega|^{2}d\mu_{g}$. 
Thus, $\omega$ is a symplectic form on $M$. 
Then from $\omega$ and $g$, an almost-complex structure $J$, which is compatible with $\omega$ and $g$, can be defined [11]. 
$J$ is called compatible with $g$ if $g(x, y)=g(Jx, Jy)$ and with $\omega$ if $\omega(x, y)=\omega(Jx, Jy)$ and $\omega(x, Jx)>0$ for all nonzero $x\in TM$. 
In this case, $(M, g, J, \omega)$ is called an almost-K\"ahler structure.

We introduce Liu's theorem on a symplectic 4-manifold which admits a metric of positive scalar curvature [13]. 
\begin{Theorem}
(Liu) Let $M$ be a smooth, compact, symplectic 4-manifold.
If $M$ admits a positive scalar curvature metric, then $M$ is diffeomorphic to either a rational or ruled surface. 
\end{Theorem}
A rational or ruled complex surface is either $\mathbb{CP}_{2}$ or ruled. 
For more explanation in detail, we refer [11].

\vspace{10pt}
We begin proof of Theorem 4. 
\begin{proof}
Suppose $(M, g)$ is orientable and $K^{\perp}<\frac{s}{4}$.
Then, there is a choice of an orientation on $M$ such that $(M, g)$ has a self-dual 2-form $\omega$ of constant length.
By multiplying constant on $\omega$, we can assume the length of $\omega$ is $\sqrt{2}$. 
Then $(M, g, \omega)$ defines an almost-K\"ahler structure. 
If $(M, g)$ is orientable and $K^{\perp}>0$, we have $K^{\perp}<\frac{s}{4}$ with any orientation. 
Thus, in this case, we can also assume $M$ is oriented in such way that $(M, g, \omega)$ defines an almost-K\"ahler structure. 
Since the scalar curvature is positive by Corollary 2, we get  then $M$ is diffeomorphic to a rational or ruled surface by applying Liu's theorem. 
Among these, only $\mathbb{CP}_{2}$ has definite intersection form. 

If $M$ is not orientable, we get the oriented double cover $\tilde{M}$ is diffeomorphic to $\mathbb{CP}_{2}$.
On the other hand, this double cover would have an orientation-reversing diffeomorphism, and therefore $\tau(\tilde{M})=0$,
which is a contradiction. 
\end{proof}

We prove Theorem 5. 
\begin{proof}
Note that the Fubini-Study metric is the unique $U(2)$-metric on $\mathbb{CP}_{2}$ with total volume $\frac{\pi^{2}}{2}$. 
This metric is Einstein and has sectional curvature $K\in [1, 4]$ [8]. 
Let us consider the space of conformal classes modulo diffeomorphism near the Funibi-Study metric $g_{0}$.
When the dimension of a manifold is greater than 2, the dimension of this space is infinite dimensional. 

Suppose $g$ be a smooth metric on $\mathbb{CP}_{2}$ which is close to $g_{0}$ in $C^{\infty}(g_{0})$-topology. 
We consider the equation $\Delta_{g}\omega_{g}=0$. 
Since $b_{+}(\mathbb{CP}_{2})=1$, the solution exists uniquely by Hodge Theory. 
Then a self-dual harmonic 2-form $\omega_{g}$ is nondegenerate 
since this form varies continuously on the space of $C^{1, \alpha}$-metrics [3]. 

It can be shown that there is a unique almost-K\"ahler metric in the conformal class $[g]$ as explained in [11]. 
We recall the proof. When metric changes from $g$ to $u^{2}g=g'$,  we have $|\omega_{g}|_{g'}=u^{-2}|\omega_{g}|_{g}$. 
Thus, if we take $u^{2}=\frac{1}{\sqrt{2}}|\omega_{g}|_{g}$, then we get $|\omega_{g}|_{g'}=\sqrt{2}$.  
Since the Hodge-star operator is a conformally invariant operator, $\omega_{g}$ is a self-dual harmonic 2-form with respect to $g'$. 
Thus, $g'$ is conformal to $g$ and $g'$ is an almost-K\"ahler metric. 

Since $\Delta_{g}\omega_{g}=0$, by elliptic estimate, we have 
\[\|\omega_{g}\|_{C^{2, \alpha}(g_{0})}\leq C\|\omega_{g}\|_{C^{0}(g_{0})}.\] 
Thus, if $g$ is $C^{2, \alpha}({g_{0}})$, then $\omega_{g}$ is also a $C^{2, \alpha}({g_{0}})$-form. 

We claim $C^{2}(g_{0})$-norm of $|\omega_{g}|_{g}$ is small. 
For this, we only consider $\|\omega_{g}-\omega_{g_{0}}\|_{C^{2, \alpha}(g_{0})}$. 
If $g$ is close to $g_{0}$ in $C^{\infty}$-topology, then $\|g-g_{0}\|_{C^{2,\alpha}(g_{0})}$ is small.
Note that we have
\[\Delta_{g}(\omega_{g}-\omega_{g_{0}})=-(\Delta_{g}-\Delta_{g_{0}})(\omega_{g_{0}}).\]
If $\|g-g_{0}\|_{C^{2,\alpha}(g_{0})}$ is small, then $C^{\alpha}$-coefficient of $(\Delta_{g}-\Delta_{g_{0}})$ is small. 
Thus, we get $\|\Delta_{g}(\omega_{g}-\omega_{g_{0}})\|_{C^{0, \alpha}(g_{0})}$ is small. 
Then by elliptic estimate, we have 
\[\|\omega_{g}-\omega_{g_{0}}\|_{C^{2, \alpha}(g_{0})}\leq C(\|\Delta_{g}\omega_{g}-\Delta_{g}\omega_{g_{0}}\|_{C^{0, \alpha}(g_{0})}+\|\omega_{g}-\omega_{g_{0}}\|_{C^{0}(g_{0})}).\]
Thus, $\|\omega_{g}-\omega_{g_{0}}\|_{C^{2, \alpha}(g_{0})}$ is small. 

Since the Levi-Civita connection is a metric connection,
we get $\||\omega_{g}|^{2}_{g_{0}}-|\omega_{g_{0}}|^{2}_{g_{0}}\|_{C^{2}(g_{0})}$ is small.
Since $\|g-g_{0}\|_{C^{2,\alpha}(g_{0})}$ is small, $\||\omega_{g}|^{2}_{g}-|\omega_{g}|^{2}_{g_{0}}\|_{C^{2}(g_{0})}$ is small. 
 Then we get  $\||\omega_{g}|_{g}-|\omega_{g_{0}}|_{g_{0}}\|_{C^{2}(g_{0})}$ is small since  $\||\omega_{g}|_{g}+|\omega_{g_{0}}|_{g_{0}}\|_{C^{2}(g_{0})}$ is bounded.

Note that $|\omega_{g_{0}}|_{g_{0}}$ is constant. 
Thus, we get the point-wise norm of $|\omega_{g}|_{g}$ is close to $\sqrt{2}$ and the point-wise norms of the first and the second derivatives of $|\omega_{g}|_{g}$ are small.

Since $g$ is close to $g_{0}$ in $C^{2}(g_{0})$-topology and the condition of positivity of sectional curvature is open in $C^{2}(g_{0})$-topology, 
we get the sectional curvature of $g$ is positive. 
Since $C^{2}(g_{0})$-norm of $u^{2}-1$ is small, we get $g'$ and $g$ are also close in $C^{2}(g_{0})$-topology. 
Thus, the almost-K\"ahler metric $g'$ has positive sectional curvature. 
\end{proof}

Finally, we consider the following result.
\begin{Theorem}
(Gursky-LeBrun [8])
If $M$ is a smooth, compact, oriented 4-manifold with strictly positive intersection form
and $g$ is an Einstein metric on $M$ with non-negative sectional curvature, then there is a diffeomorphism $\Phi:M\to \mathbb{CP}_{2}$
such that $g=\Phi^{*}cg_{FS}$, where $g_{FS}$ is the Fubini-Study metric on $\mathbb{CP}_{2}$ and $c$ is some positive constant. 
\end{Theorem}
In [8], it was shown that the assumptions in Theorem 7 implies $W_{-}=0$
and then Hitchin's theorem [7], [9] was applied. 
In the context of this theorem, we add an Einstein condition in our case, in fact, a weaker condition by LeBrun [10]. 
Sekigawa's theorem implies that an almost-K\"ahler Einstein metric with non-negative scalar curvature is K\"ahler-Einstein [18].
LeBrun generalized this result by using weaker conditions in the following way. 

\begin {Theorem}
(LeBrun [10])
Let $(M, g, J, \omega)$ be an almost-K\"ahler 4-manifold. 
Suppose $s\geq 0$ and $\delta W^{+}=0$. 
Then $g$ is a constant scalar curvature K\"ahler metric. 
\end{Theorem}
Note that if $g$ is Einstein, $\delta W^{+}=0$ holds. Thus, this theorem generalizes Sekigawa's theorem. 

\begin{Corollary}
Let $(M, g)$ be an oriented, compact riemannian 4-manifold with the signature $\tau>0$. 
Suppose $(M, g)$ admits a harmonic 2-form of constant length. 
If $g$ has $K^{\perp}>0$ or $K^{\perp}<\frac{s}{4}$ and $\delta W^{+}=0$, 
then there is a diffeomorphism $\Phi: M\to \mathbb{CP}_{2}$ such that $g=\Phi^{*}cg_{FS}$,
 where $g_{FS}$ is the Fubini-Study metric on $\mathbb{CP}_{2}$ and $c$ is some positive constant.
\end{Corollary}
\begin{proof}
By Theorem 4, $(M,g)$ defines an almost-K\"ahler structure with $s>0$  and $\delta W^{+}=0$. 
Then, $g$ is a constant scalar curvature K\"ahler metric by Theorem 8. 
Moreover, $M$ is diffeomorphic to $\mathbb{C}P^{2}$. 
By Yau's theorem, if a compact complex surface is homotopy equivalent to $\mathbb{C}P^{2}$, 
then it is biholomorphic to $\mathbb{CP}_{2}$ [19].  
Thus, $M$ is biholomorphic to $\mathbb{CP}_{2}$, which we denote by $\Psi: \mathbb{CP}_{2}\to M$. 
We claim that a constant scalar curvature K\"ahler metric on $\mathbb{CP}_{2}$ is K\"ahler-Einstein. 
First, note that $b_{+}(\mathbb{CP}_{2})=1$ and therefore the K\"ahler form $\omega$ is the unique representative harmonic 2-form up to rescaling. 
On the other hand, the Ricci form $\rho$ is a closed $(1, 1)$-form. 
The self-dual part of $\rho$ is given by $\rho^{+}=\frac{\rho+*\rho}{2}$. 
Since $d\rho=0$, we get $d\rho^{+}=0$ if and only if $d*\rho=0$. 
In other words, $\rho$ is a harmonic 2-form if $d\rho^{+}=0$. 
Note that $\rho^{+}$ is given by $<\rho, \omega>\omega=\frac{s}{4}\omega$. 
Since the scalar curvature $s$ is constant, we get $\rho$ is a harmonic 2-form. 
Thus, on $\mathbb{CP}_{2}$, $\rho$ is equal to $\omega$ up to constants and $\omega$ is a K\"ahler-Einstein metric. 
Then by Matsushima theorem [14], we get a K\"ahler-Einstein metric on $\mathbb{CP}_{2}$ is the 
Fubini-Study metric up to automorphisms of $\mathbb{CP}_{2}$ and rescaling. 
Thus, we get that $g$ on $M$ is given by $\Phi^{*}cg_{FS}$ for some biholomorphism $\Phi: M \to \mathbb{CP}_{2}$ and for some positive constant $c$. 
\end{proof}

\begin{Remark}
Note that we used the fact that a constant scalar curvature K\"ahler metric on $\mathbb{CP}_{2}$ is K\"ahler-Einstein. 
However, Lichnerowitz proved that an analogue of Matsushima's theorem holds for a constant scalar curvature K\"ahler metric [12]. 
Thus, more carefully we may apply Lichnerowitz theorem since $b_{1}=0$ on $\mathbb{CP}_{2}$. 
\end{Remark}

\begin{Remark}
In Corollary 3, $\Psi^{*}g$ is a K\"ahler-Einstein metric on $\mathbb{CP}_{2}$. 
Thus, instead of Matsushima theorem, we may apply Bando-Mabuchi theorem, which implies 
that $\Psi^{*}g$ is the Fubini-Study metric on $\mathbb{CP}_{2}$ up to automorphisms [1]. 
\end{Remark}

\begin{Corollary}
Let $(M^{4}, g)$ be a compact 4-manifold with a harmonic 2-form of constant length. 
If $g$ is an Einstein metric with $K^{\perp}>0$ or $K^{\perp}<\frac{s}{4}$, then $M$ is diffeomorphic to $\mathbb{C}P^{2}$ 
and there is a diffeomorphism $\Phi:M\to \mathbb{CP}_{2}$
such that $g=\Phi^{*}cg_{FS}$, where $g_{FS}$ is the Fubini-Study metric on $\mathbb{CP}_{2}$ and $c$ is some positive constant.
\end{Corollary}

\begin{proof}
From Theorem 4, we get $(M, g)$ can be oriented so that $(M, g)$ defines an almost-K\"ahler structure. 
With this orientation, we have $\delta W_{+}=0$ since $\delta W=0$ if $g$ is Einstein. 
Thus, we can apply Corollary 3. 
Or knowing that $(M, g)$ can be oriented so that the intersection form is positive from Proposition 1, 
we get the conclusion by using Theorem 7.  
Note that in this case, we do not use Liu's theorem in order to get $M$ is diffeomorphic to $\mathbb{CP}_{2}$. 
\end{proof}

\begin{Remark}
We note that Gursky-LeBrun's theorem implies, in particular, there is no other Einstein metric with nonnegative sectional curvature on $\mathbb{CP}_{2}$
except the Fubini-Study metric up to isometry and rescaling. 
\end{Remark}

\vspace{20pt}

Email address: kiysd@snu.ac.kr


\vspace{50pt}

\begin{thebibliography}{9}
\bibitem{BM}
S. Bando and T. Mabuchi, 
\emph{Uniqueness of Einstein K\"ahler metrics modulo connected group actions}, 
Algebraic geometry, Sendai, 1985, 11-40, Adv. Stud. Pure Math., 10, North-Holland, Amsterdam, 1987. 
\bibitem{B}
C. B\"ar, 
\emph{Geometrically formal 4-manifolds with nonnegative sectional curvature}, 
Commun. Anal. Geom. 23 (2015), 479-497. 
\bibitem{BL}
C. J. Bishop and C. LeBrun, 
\emph{Anti-Self-Dual 4-Manifolds, Quasi-Fuchsian Groups, and Almost-K\"ahler Geometry}, 
e-print arXiv:1708.03824 [math. DG].
\bibitem{CR}
E. Costa and E. Ribeiro Jr., 
\emph{Four-Dimensional Compact Manifold with Nonnegative Biorthogonal Curvature},
Michgan Math. J. 63 (2014), 747-761.
\bibitem{D}
S. Donaldson, 
\emph{An application of gauge theory to four dimensional topology},
J. Differential Geom. 18 (1983), 279-315.
\bibitem{F}
M. Freedman,
\emph{The topology of four dimensional manifolds}, 
J. Differential Geom. 17 (1982), 357-454.
\bibitem{FK}
T. Friedrich and H. Kurke,
\emph{Compact Four-Dimensional Self-Dual Einstein Manifolds with Positive scalar curvature}, 
Math. Nachr. 106 (1982) 271-299. 
\bibitem{GL}
M. Gursky and C. LeBrun, 
\emph{On Einstein Manifolds of Positive Sectional Curvature}, 
Ann. Glob. An. Geom. 17 (1999), 315-328.
\bibitem{H}
N. J. Hitchin, 
\emph{K\"ahlerian Twistor Spaces}, 
Proc. Lond. Math. Soc. 43 (1981) 133-150. 
\bibitem{L}
C. LeBrun, 
\emph{Einstein Metrics, Harmonic Forms, and Symplectic Four-Manifolds}, 
Ann. Global. An. Geom. 48 (2015), 75-85. 
\bibitem{L}
C. LeBrun, 
\emph{Weyl Curvature, Del Pezzo Surfaces, and Almost-K\"ahler Geometry}, 
J. Geom. Analysis 25 (2015), 1744-1772. 
\bibitem{L}
A. Lichnerowicz
\emph{Isom\'etries et transformations analytiques d'une vari\'et\'e K\"ahl\'erienne compacte. (French)},
Bull. Soc. Math. France 87, 1959. 427-437. 
\bibitem{aL}
A-K. Liu,
\emph{Some new applications of general wall-crossing formula, Gompf's conjecture and its applications},
Math. Res. Lett, 3 (1996), pp.569-585.
\bibitem{M}
Y. Matsushima, 
\emph{Sur la structure du groupe d' hom\'eomorphismes analytiques d'une certaine vari\'et\'e K\"ahl\'erienne. (French)},
Nagoya Math J. 11 (1957) 145-150. 
\bibitem{N}
M. Noronha, 
\emph{Some results on nonnegatively curved four manifolds}, 
Mat. Contemp. 9 (1995), 153-175.
\bibitem{S}
W. Seaman, 
\emph{Two-forms on four manifolds}, 
Proc. Amer. Math. Soc. 101 (1987), 353-357.


\bibitem{S}

W. Seaman, 
\emph{Some remarks on positively curved 4-manifolds},
Michigan Math. J. 35 (1988), 179-183.

\bibitem{S}
K. Sekigawa, 
\emph{On some 4-dimensional compact Einstein almost K\"ahler manifolds}, 
Math. Ann., 271 (1985), pp. 333-337.


\bibitem{Y}
Yau,
\emph{On Calabi's conjecture and some new results in algebraic geometry}, 
Proc. Nat. Acad. Sci. USA 74, 1798-1799 (1977). 

\end{thebibliography}
\end{document}